\documentclass[11pt]{article}

\usepackage{epsfig}
\usepackage{epstopdf}
\usepackage{subfigure}
\usepackage[font=scriptsize]{caption}
\usepackage{graphicx}
\graphicspath{{figures/}}

\usepackage{amsmath}
\usepackage{amsthm,amscd,amsfonts}
\usepackage{amssymb, upref, color}
\usepackage{amsmath,amssymb,amsthm}
\usepackage{color}
\usepackage[colorlinks]{hyperref}
\usepackage{graphicx}
\usepackage{epsf,epsfig,subfigure, verbatim}
\usepackage{latexsym,bm}
\usepackage{enumerate}
\usepackage{listings}
\usepackage[numbers,sort&compress]{natbib}

\usepackage{mathrsfs}
	\usepackage{color}

	\numberwithin{equation}{section}

	\theoremstyle{plain}
	\newtheorem{exam}{Example}[section]
	\newtheorem{theorem}[exam]{Theorem}
	\newtheorem{lemma}[exam]{Lemma}

\begin{document}

	\sloppy
	\captionsetup[figure]{labelfont={bf},name={Fig.},labelsep=period}
	\captionsetup[table]{labelfont={bf},name={Table},labelsep=period}

	\title{ Bifurcation from a blood flow with variable body force
\footnote{  This paper was jointly supported from Natural Science Foundation of Zhejiang Province (No. LZ24A010006, LZ23A010001), National Natural Science Foundation of China (No. 12571165).}	}
	\author
	{
		Yuchao He$^{a}$\,\,\,\,\,Yongli Song$^{b}$\,\,\,\,
		Yonghui Xia$^{c}\footnote{Corresponding author.   yhxia@zjnu.cn;xiadoc@163.com }$
		\\
		{\small \textit{$^a$ School of Mathematical  Science,  Zhejiang Normal University, 321004, Jinhua, China}}\\
		{\small Email:  YuchaoHe@zjnu.edu.cn; xiadoc@163.com; yhxia@zjnu.cn}
		\\
		{\small \textit{$^b$ School of Mathematics,  Hangzhou Normal University,  311121,  Hangzhou, China}}\\
		{\small Email: songyl@hznu.edu.cn}\\
		{\small \em c. School of Mathematics, Foshan University, Foshan, 528000, China.}\\
	}
	\maketitle
	\begin{abstract}
This paper investigates the existence of periodic solutions in blood flow propagating through vessels with free boundary conditions via the bifurcation theory. 
It is rigorously proved that  a local $C^1$-curve of small-amplitude periodic solutions is bifurcated. In contrast to previous studies on periodic flows that primarily focus on constant vorticity, our work emphasizes the bifurcation analysis of periodic solutions in blood flow with harmonic vorticity and external body forces. To utilize Crandall-Rabinowitz bifurcation theorem, the fundamental challenge lies in reducing a multiple variable-PDE subject to free boundary conditions to a system of one variable-ODE with fixed boundary conditions. 
	\end{abstract}
	
	
	\section{Introduction}


	The study of human arterial blood flow model has a long history, it tracks back to the original work by  Euler \cite{euler} in 1775. He proposed a one-dimension inviscid flow with the conservation of mass and momentum. Euler's work was innovative, but regrettably failed  in considering the wave-like
	nature of the flow and finding a solution for his equations.  After 100 years later, Anliker et al.  \cite{anliker} and Skalak \cite{skalak} applied  Riemann's  characteristics method (see \cite{riemann}) to develop   the research on the artistic flow.	
	Up till to 21-th century, the study of the blood flows has once again  attracted widespread attentions. In 2002,  Quarteroni and  Formaggia \cite{book} studied the problem of haemodynamics and 2D, 3D blood flow models. For the convenience of research, the model of blood flow was often seen as an incompressible non-Newtonian fluid with viscosity. 
In 2015, Friedman \cite{Friedman} et al. considered a free boundary problem for small plaques in the
	artery and  studied the stability of blood flow system with some conditions.
	In 2016, Quarteroni et al. \cite{AAC} used the geometric multiscale approach to deal with the 3D  blood flow problem and emphasize the need for rigorous mathematical models to fill the gap between theory and practice.
	In order to better simulate the dynamic behavior in actual blood flow, Tabi et al. \cite{Tabi} considered  the  impact of  viscosity, radiations, magnetic fields on velocity and temperature on the blood flow.
	And in 2021, Bi et al. \cite{Bi} studied one kind of incompressible blood flow under the weak viscosity condition. A traveling wave solution and some
	properties of it are  given.
	The numerical algorithms on blood flow model   have also attracted a large number of scholars to conduct research,  one can refer  to \cite{  JCP3,striva}. In the systematic monograph on the mathematical modelling
	of the cardiovascular
	system \cite{book}, the authors mentioned that a few  of factors should be incorporated into the blood vessel artery
	based on the basic model. These factors includes axial symmetry, radial displacements, constant pressure, dominance of axial velocity and body forces. Quarteroni and  Formaggia \cite{book} incorporated the body force into the blood model described by a viscous and  incompressible Navier-Stokes equation. Recent years have witnessed significant advances in the study of traveling wave solutions to free boundary flow problems. Within the framework of the full incompressible Navier-Stokes equations, Koganemaru and Tice \cite{tice1} established the existence of traveling wave solutions with Navier slip boundary conditions; subsequently, Stevenson and Tice \cite{tice2} further investigated gravity-driven traveling bore wave solutions. In parallel, for the more simplified shallow water model, Stevenson and Tice systematically analyzed the well-posedness of the traveling wave problem \cite{tice3} and the asymptotic behavior of two-dimensional stationary wave solutions \cite{tice4}.
For more detailed applications of the body force, one can refer to \cite{Foias}.
	However, most of the above mentioned works are on the numerical computation of the blood flow  model.
	The problem of solving the dynamic equations of blood flow with boundary conditions is rarely solved by scholars. It is known that  if the surface of flow is considered as a free boundary, solving the PDEs becomes particularly complex. An effective way to describe the  fluid dynamics is to  
	find laminar solutions (see \cite{chu1,chu2,chu3,AAC}).
	In this paper, we use one kind of generalized Navier-Stokes  equations with viscosity terms and free boundary conditions to describe the dynamic behavior of blood flow in arteries. In particular, we note that the volume force in the blood flow is not constant, but changes with the position and time (For more details on the variable body force, readers can  refer to \cite{Foias}). We pay our particular attentions to the rotational blood flow. 
	The blood model with  time dependent body force is formulated as follows:
\begin{equation}\label{eular}\begin{cases}
			u_t+uu_x+vu_y-\mu\Delta u=-\frac{1}{\rho}P_x+f_1(x-ct,y),\\
			v_t+uv_x+vv_y-\mu\Delta v=	-\frac{1}{\rho}P_y-f_2(x-ct,y),
		\end{cases}
	\end{equation}
where $(u(x, y), v(x, y))$  denotes the blood velocity field,  $\mu$ is the viscosity coefficient, $\rho$ is the constant density, $P$ is the pressure, $f_1$ and $f_2$ are non-negative body forces.

	By introducing the flow force function and height transformation, we obtain a class of equivalent equations and corresponding exact laminar solutions. Further, according to  the bifurcation theorem in \cite{c-r}, local bifurcation results have been presented. In particular, we prove the existence of  a local  $C^1$-curve  of
	small-amplitude solution.
The aim of this paper is to present the qualitative behaviors of the blood dynamics. Motivated by the works of 
 Chu \cite{chu1}, 
we  develop a equivalent reformulation and  find the exact solutions, which is interesting from the point of mathematics. This helps us easier to observe the dynamics of the blood flow, consequently, it is beneficial for the treatment of cardiovascular diseases.  It is believed that our method sheds some new light on the study of the blood flow.
	Our reformulation of the blood flow would benefit to design
	numerical  algorithms for the blood flow model. It is noteworthy that the majority of previous studies on steady periodic flows have predominantly focused on scenarios with constant vorticity (e.g., see \cite{chu1,chu2,chu3,wang1,wang2}). In contrast, this paper investigates the bifurcation of periodic solutions in blood flow characterized by harmonic vorticity and body force.

	The outline of the rest of this paper: in Section 2,
	we state a blood flow model with variable body forces, and the problem is equivalent to
	a second order quasi-linear elliptic equation by some transformations. Finally, in Section 3,
	the existence of  a local  $C^1$-curve  of
	small-amplitude solution is strictly proved by Crandall-Rabinowitz bifurcation theorem.


	\section{Equivalent reformulation of blood flow model}
	\subsection{Preliminaries and assumptions}
	In the vascular system, choose $(x,y)\in\Phi\subset \mathbb{R}^2$ and let the $x$-axis be the axial direction of blood vessels and the $y$-axis be the radial direction of blood vessels. In actual vascular fluid models, there are usually two types of situations:  the blood vessels are fully filled with blood; the  blood vessels are not fully filled with blood.	
	If the blood vessels are not fully filled with blood,  the free surface, denoted by $y = \Omega(x,t)$, on the domain bounded above   satisfies
	\begin{equation}\label{free surface}\int_{-\pi}^{\pi}\Omega(x,t)dx=0.\end{equation}
	
	And below by the central axis of blood vessels $y = -d$ with $0 <d< \infty$, the period respect to $x$ is chosen as $2\pi$.
	The situation when the blood flow fully fills the blood vessels can be seen as a special case of the above situation and  $\Omega(x,t)$ reduces to $0$. Without loss of generality, we  only consider the case when the blood vessel is not fully filled in the following (see figure 1).
	\begin{figure}[htbp]
		\begin{minipage}[t]{0.45\textwidth}
			\centering
			\includegraphics[width=\textwidth]{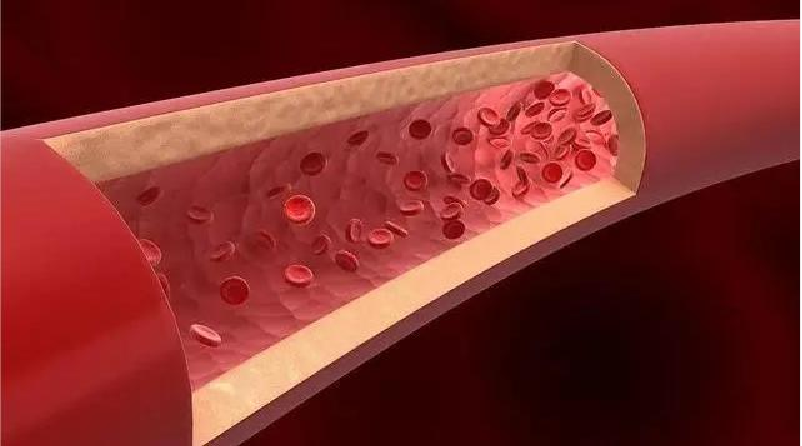}
			\label{fig:img1}
		\end{minipage}\hfill
		\begin{minipage}[t]{0.45\textwidth}
			\centering
			\includegraphics[width=\textwidth]{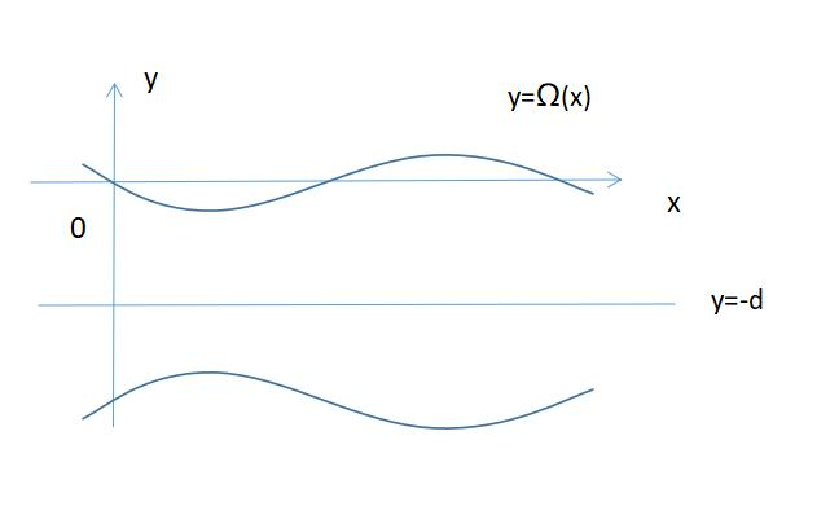}
			\label{fig:img2}
		\end{minipage}
		\caption{Cross section of blood flow in blood vessels.}
		\label{fig:side_by_side}	
	\end{figure}

	Assume that the speed $c$ of traveling waves is positive ($c>0$). That is, the free surface $\Omega(x,t)$ takes the form of $\Omega(x,t) = \Omega(x - ct)$.

And  $(u(x, y), v(x, y))$ is used to  denote the blood velocity field in the vascular region $\{(x, y) \in R^2 : -d < y < \Omega(x)\}$. Set that the blood flow has no
	stagnation points, i.e.
	\begin{equation}\label{no	stagnation}
		u-c<0\quad  {\rm for}\quad   -d\leq y\leq\Omega(x),
	\end{equation}
	through blood vessels.  In this paper, we consider the blood flow model \eqref{eular}
	with the condition of mass conservation
	\begin{equation}\label{mass conservation}
		u_x+v_y=0.
	\end{equation}
	The movement of blood flow satisfies the following boundary conditions:
	\begin{equation}
		v=0\quad \rm{on} \ \ y=-d,
	\end{equation}
	and
	\begin{equation}
		P=P_0\quad {\rm on} \ \ y=\Omega(x).
	\end{equation}
	If the blood flow does not fully fill the blood vessels, $v$ satisfies
	\begin{equation}\label{y=eta2}
		v=(u-c)\Omega'(x)\quad \rm{on}\ \  y=\Omega(x).
	\end{equation}
	Otherwise,
	\begin{equation}\label{y=eta}
		v=0\quad \rm{on}\ \  y=\Omega(x).
	\end{equation}
	In the above equations, $\mu$ is the viscosity coefficient, $\rho$ is the constant density, $P$ is the pressure, $f_1$ and $f_2$ are non-negative body forces satisfying the following condition
	\begin{equation}\label{f1f2}
		\frac{\partial f_1}{\partial y}=-\frac{\partial f_2}{\partial x}.
	\end{equation}
	And we assume that for any $x$, the body force $f_1$ in vertical direction satisfies
	\[
	\int_{-d}^{\Omega(x)}f_1(x,y)dy=0,
	\]
	$f_1(x,y)$ vanishes on the boundary: $f_1(x,\Omega(x))=f_1(x,-d)=0.$\\
	We  define the  stream function  $\Phi(x,y)$: $\Phi\to\mathbb R$ by as follows
	\begin{equation}\label{stream}
		\begin{cases}
			\Phi_x=-v,\\
			\Phi_y=u-c.
		\end{cases}
	\end{equation}
	Moreover, $\Phi$ is normalized by choosing $\Phi=0$ on $y=\Omega(x)$ and then $\Phi=m$ on $y=-d$,
	where
	$$m=\int_{\Omega(x)}^{-d}(u(x,y)-c)dy.$$
	It is easy to see that $m$ is a constant by taking the derivative with respect to $ x$. $m$ varies with the solution of \eqref{eular}-\eqref{y=eta}.
	Vorticity  is an important factor in characterizing fluid dynamics behavior. In viscous fluids, since the inseparable relationship between vorticity and viscosity, the vorticity  is often used to study the dynamic behavior of flows \cite{vorticity}.
	Unlike previous studies, our study examines the characteristic flow behavior induced by harmonic vorticity conditions, namely,
	\begin{equation}\label{non-rotation}
		\omega=u_y-v_x,
	\end{equation}
and {
\begin{equation}\label{harmonic}
	\Delta \omega=0.
\end{equation}}
	Combining \eqref{non-rotation} with \eqref{mass conservation} and  taking the traveling wave transform
	$(x-ct, y)\to(x, y)$, we deduce that \eqref{eular} is equivalent to
	{
\begin{equation}\label{stokes}\begin{cases}
		(u-c)u_x+vu_y-\mu\omega_y=-\frac{1}{\rho}P_x+f_1(x-ct,y),\\
		(u-c)v_x+vv_y+\mu\omega_x=	-\frac{1}{\rho}P_y-f_2(x-ct,y),
	\end{cases}
\end{equation}
	since
	\[
	\begin{split}
	\Delta u&= u_{xx}+u_{yy}
	=(-v_y)_x+u_{yy}
	=(u_y-v_x)_y
	=\omega_y
	\end{split}
	\]
	and
	\[
	\Delta v=v_{xx}+v_{yy}=v_{xx}+(-u_x)_y=(v_x-u_y)_x=-\omega_x.
	\]}
	\subsection{Equivalent reformulations}
	Distinctly, according to \eqref{stokes}, the total energy has the following expression{
	\[
	E=\frac{\psi_x^2+\psi_y^2}{2}+\frac{P}{\rho}+\Gamma(\psi)+F+M,
	\]
	where \[F_x=-f_1,\ F_y=f_2,\  M_x=-\mu\omega_y,\  M_y=\mu\omega_x\]
	and
	\[
	\Gamma(p)=\int_0^{p}\gamma(s)ds,
	\]
	with $\Gamma_{max}=\max\limits_{m<p<0}\Gamma(p)\geq0$. And $\gamma$ denotes the vorticity function, which will be explained in detail in the following. The existence of $F(x,y)$ and $M(x,y)$ follows from \eqref{f1f2} and \eqref{harmonic}. }
	
	Based on the analysis above, we obtain that equations \eqref{mass conservation}-\eqref{y=eta} and \eqref{stokes} are
	transformed into the following equations,
\begin{equation}\label{psi}
	\begin{cases}
		{\Delta\psi=\omega,\quad }&	{{\rm for}\quad -d<y<\Omega(x)},\\
		{(\nabla\psi)+2(F+M)=Q,\quad} &	{{\rm on} \quad y=\Omega(x),}\\
		\psi=0, \quad &{\rm on}\quad y=\Omega(x),\\
		\psi=m, \quad&{\rm on}\quad y=-d,
	\end{cases}
\end{equation}
where $Q=2(E-\frac{P_0}{\rho})$ is constant i.e. $\frac{\partial Q}{\partial x}=\frac{\partial Q}{\partial y}=0.$

Let $w=x$, $z=\frac{\psi}{m},$ then
\begin{equation}\label{xy}
	\frac{\partial}{\partial x}=\frac{\psi_x}{m}\cdot\frac{\partial}{\partial z}+\frac{\partial}{\partial w},\quad \frac{\partial}{\partial y}=\frac{\psi_y}{m}\cdot\frac{\partial}{\partial z}.
\end{equation}
\begin{figure}[h]
	\centering
	\includegraphics[scale=0.8]{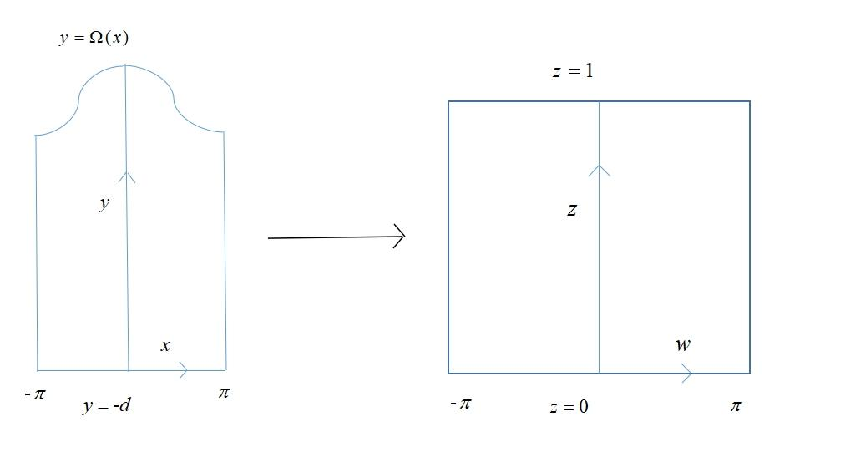}
	\caption{Sketch of the blood flow region transformation.}
	\label{figure}
\end{figure}
 It is worth mentioning that under the new variables, the original free
boundary blood flow region is transformed into a new fixed boundary region, greatly simplifying the
problem. And it is intuitive in the graph 2.

Note that $\omega$ is a function independent of $w$, since 
\begin{equation}\label{w=0}
\frac{\partial}{\partial w}\omega=(\frac{\partial}{\partial x}+\frac{v}{u-c}\frac{\partial}{\partial y})\omega=0.
\end{equation}
There exists a vorticity function $\gamma(p)=\omega$.

Take the height function as
\[
h=y+d.
\]
Based on the definition of height function $h$, we have
\[
\int_{-\pi}^{\pi}h(w,0)dw=2\pi d
\]
and\[
\frac{\partial h}{\partial q}=-\frac{\psi_x}{\psi_y},\quad\frac{\partial h}{\partial p}=\frac{m}{\psi_y}.
\]
Then \eqref{psi} is transformed into
\begin{equation}\label{h}
	\begin{cases}
		{(1+h_w^2)h_{zz}-2h_wh_zh_{zw}+h_{ww}h_z^2=-\frac{\omega(z)}{m}h_z^3,}\\
		{1+h_w^2=\frac{1}{m^2}[Q-2(F+M)]h_z^2,\quad }&	{{\rm on}\quad z=0,}\\
			{h=0,\quad }&	{{\rm on}\quad z=1,}\\
		{	\int_{-\pi}^{\pi}h(w,z)dw=2\pi d,\quad }	&	{{\rm on}\quad z=0.}
		\end{cases}
\end{equation}
In view of \eqref{f1f2} and \eqref{harmonic}, the sum of $M$ and $F$ is expressed by the following  integration which  is independent of path selection.
{
\begin{equation}\label{lujing}
\begin{split}
F(x,y)+M(x,y
)=&\int_{-d}^{h-d}(f_2(w,r)+\mu\omega_w(w,r))dr+\int_{0}^{w}(-f_1(r,-d)-\mu\omega_z(r,-d))dr\\
&+F(0,-d)+M(0,-d).
\end{split}
\end{equation}}

\section{Main result}
\subsection{Laminar solution}

	Firstly, we seek the laminar solutions $H(z)$  to  \eqref{h}, which is independent of $w$.\\
Note that $F$ is the
first integral, it can be corrected by adding or subtracting a constant. Without loss generality, we
let $F(0, -d) = 0$, or else, if $F(0,-d) = {\rm constant}\ne 0$, take $\widetilde{F} = F- {\rm constant}$. Similarly, we assume that $M(0,-d)=0.$ Since we assume $f_1(x,\Omega(x))=0$, we have  $\frac{dF(x,\Omega(x))}{dx}=f_1(x,\Omega(x))=0$ and $F$ is independent of $w$ on $y=\Omega(x)$ or on $z=0$. Set $F_1(y)=F(w,z)\bigg|_{z=0}$,  moreover  $F_1(y)|_{y=\Omega(x)}=F_1(H(0)-d).$\\
In this paper, We assume that the variation of vorticity on the bottom can be ignored, which implies $\int_0^w \mu\omega_z(r,-d)dr=0$. Then combining  \eqref{w=0} with \eqref{lujing}, we obtain 
\[
\begin{split}
[F(x,y)+M(x,y)]\bigg|_{y=\Omega(x)}&=
[F(w,z)+M(w,z)]\bigg|_{z=0}\\
&=\int_{-d}^{h-d}f_2(w,r)dr+\int_{0}^{w}-f_1(r,-d)dr+F(0,-d)\\
&=F_1(H(0)-d).
\end{split}
\]
Consequently, $H(z)$ solves
\begin{equation}
	\begin{cases}
	{H_{zz}=-\frac{\omega(z)}{m}H_z^3,\quad} &{{\rm for}\quad 0<z<1,}\\	{1=\frac{1}{m^2}[Q-2F_1(H(0)-d)]H_z^2,\quad} &{{\rm on} \quad z=0,}\\
	H=0,\quad&{\rm on}\quad z=1,\\
H=d,\quad&{\rm on}\quad z=0.	
\end{cases}
\end{equation}
According to the equations above, we obtain
\begin{equation}
{	H_z=\frac{1}{\sqrt{\xi-2\Gamma(z)}}}
\end{equation}
and the laminar solution of \eqref{h}
\begin{equation}\label{H}
{	H=\int_0^z\frac{1}{\sqrt{\xi-2\Gamma(s)}}ds+d+F_1^{-1}(\frac{Q-\xi m^2}{2}),}
\end{equation}
where { $\xi=\frac{1}{H_z^2(0)}.$}\\
The existence of $F_1^{-1}$ in \eqref{H} follows from the nonnegativity of $f_1$ and $f_2.$\\
And according to the boundary condition $H=d$ on $z=0$ in \eqref{H}, we obtain
\begin{equation}\label{Q}
Q=\xi m^2.
\end{equation}
\subsection{Linearised problem}
Now we seek solutions $h\in  C_{per}^{3,\alpha}$ of the form
\[
h(w,z)=H(z,\xi)+\epsilon g(w,z).
\]
Denote $a(\xi,z)=\sqrt{\xi-2\Gamma(z) }$. In this paper, we assume that there exist a positive constant  $c_4$ such that  $\xi-2\Gamma\geq c_4>0$ to ensure that  $a$ exists.\\
Note that $F_1(x,y)$ has the follow Taylor expansion.
\[
\begin{split}
F(x,y)&=F(0,-d)+\int_{-d}^{h-d}f_2(x,r)dr-\int_0^xf_1(r,-d)dr\\
&=\int_{-d}^{\epsilon f}f_2(x,r)dr\\
&=\int_{-d}^0f_2(x,r)dr+f_2(x,0)f\epsilon+o(\epsilon^2).
\end{split}
\]
Then we obtain that
\begin{equation}\label{f}
	\begin{cases}
		(a^3g_z)_z+ag_{ww}=0,\\
{a\left(Q-2\int_{-d}^0f_2(x,r)dr\right)g_z=f_2(x,0)g,\quad}&{{\rm on}\quad z=0,}\\
g=0,\quad&{\rm on} \quad z=1,
		\end{cases}
\end{equation}
where { $a=\sqrt{\xi-2\Gamma(p)}.$}\\
Note that $g$ is an even function, we have the following equation by using Fourier transform.
\begin{equation}
	g(z,w)=\sum_{k=0}^\infty g_k(z)\cos (kw) \in  C^2_{per}(\bar R_z).
\end{equation}
with the coefficients
\[
g_0(z)=\frac{1}{2\pi}\int_{-\pi}^{\pi}g(w,z)dz\quad {\rm and}\quad \ g_k(z)=\frac{1}{2\pi}\int_{-\pi}^{\pi}g(w,z)cos(kw)dw,\ k\geq 1.
\]
Obviously, $g$ is a solution to \eqref{f} iff for $M$ the following Sturm-Liouville problem is solved by each $g_k \ne 0$.

\begin{equation}\label{3.5}
	\begin{cases}
		(a^3{M}_z)_z=k^2aM,\quad &{\rm for}\quad 0<z<1,\\
	{	a\left(Q-2\int_{-d}^0f_2(x,r)dr\right){M}_z=f_2(x,0)M,\quad}&{{\rm on}\quad z=0,}	\\
		M=0,\quad &{\rm on}\quad z=1.
	\end{cases}
\end{equation}
As a special case, if $k = 0$, the solution $M_0(z)$ of following equations is zero.

\begin{equation}\label{M0}
	\begin{cases}
		(a^3{M_0}_z)_z=0,\quad&{\rm for}\quad 0<z<1,\\
		{a\left(Q-2\int_{-d}^0f_2(x,r)dr\right){M_0}_z=f_2(x,0)M_0,\quad}&{{\rm on}\quad  z=0,}	\\
		M_0=0,\quad &{\rm on }\quad z=1.
	\end{cases}
\end{equation}
{Denote
\[
\mu(\xi)=G(\xi)=\inf_{\zeta\in H^1(0,1),\ \zeta(0)=0,\ \zeta\not\equiv0 }\frac{-\frac{a^3\zeta^2(0)f_2(x,0)}{\xi(Q-2\int_{-d}^0f_2(x,r)dr)}+\int_0^1a^3\zeta_z^2dz}{\int_0^1a\zeta^2dz},
\]}
where $H^1$ is a Hilbert space.
Obviously, a minimizer $M \in C^{3,\alpha}[0, 1]$ of \eqref{3.5} is a classical solution of following equation.
\begin{equation}\label{Mk}
	\begin{cases}
		(a^3M_z)_z=-\mu(\xi)aM,\quad &{\rm  for}\quad 0<z<1,\\
		{\xi\left(Q-2\int_{-d}^0f_2(x,r)dr\right)M_z=f_2(x,0)M, \quad }&{{\rm on}\quad z=0,}\\
		M=0,\quad&{\rm on}\quad z=1.
	\end{cases}
\end{equation}
{We ensure that the solutions to \eqref{3.5} exist, if $\mu(\xi)=-1 $ holds
for some $\xi>0.$
Set $c_1=f_2(x,0)>0$, $c_2=\int_{-d}^0f_2(x,r)dr>0$, $c_3=\frac{2\Gamma_{\max}(Q-2c_2)}{c_1-Q+2c_2}>0$. In this paper, we assume that $c_1, c_2, c_3$ are constants such that $c_1+2c_2>Q>2c_2$  and $Q>2c_3+2\sqrt3c_2$. \\
For the existence of \eqref{3.5}, it is necessary that $\mu(\xi) =-1$ for some $\xi>2\Gamma_{\max}$.\\
First, we prove $\mu\geq-1$, for $2\Gamma_{\max}+c_3\leq\xi\leq-2\Gamma_{\max}+c_3^{\frac{4}{3}}.$
\[
\begin{split}
\int_0^1(a^3\zeta_z^2+a\zeta^2)dz&\geq c_3\frac{1}{2}\int_0^1(a^2\zeta_z^2+\zeta^2)dz\\
&\geq
2c_3\bigg|\int_0^1\zeta_z\zeta dz\bigg|\\
&\geq
c_3\zeta(0)^2\\
&\geq\frac{a^3c_1\zeta^2(0)}{\xi(Q-2c_2)},
\end{split}
\]
which implies
\[
\mu(\xi)\geq-1,
\]
for $2\Gamma_{\max}+c_3\leq\xi\leq-2\Gamma_{\max}+c_3^{\frac{4}{3}}$.}

Moreover, using the similar
method in \cite{c-r,JCP3}, we can show that $\mu(\xi)$ is continuous and a strictly increasing function of
$\xi$ in any interval where it is negative and the solution $\xi^*$ to $\mu(\xi) = -1$ is unique if exists. {Note that we have proved $\mu(\xi)\geq-1$, for $2\Gamma_{\max}+c_3\leq\xi\leq-2\Gamma_{\max}+c_3^{\frac{4}{3}}.$ Therefore, if we are able to show that $\mu(\xi) \leq-1$ for some $\xi > 2\Gamma_{\max}$, then we know that there exists
some $\xi>2\Gamma_{\max}$	 such that $\mu(\xi) = -1$.
Now we prove the following sufficient condition for the vorticity function to ensure that
$\mu(\xi) \leq -1$ for some  $\xi>2\Gamma_{\max}$.
Now we prove that $\mu(\xi)\leq-1$ holds for some $2\Gamma_{\max}+\frac{c_3}{2}>\xi>2\Gamma_{max}$ with a sufficient condition.}{
\begin{lemma}\label{lemma1}
If there exists $\beta\in (3,\infty]$ such that
\[
\frac{c_4^{\frac{3}{2}}c_1}{(2\Gamma_{\max}+\frac{c_3}{2})(Q-2c_2)}>\frac{\sqrt{ 2}}{3}\beta^*||\gamma||_{\beta}^{\frac{2}{3}}p_1^{\frac{3}{2\beta^*}-1}+\frac{2\sqrt2\beta^*}{1+4\beta^*}||\gamma||^{\frac{1}{2}}_\beta p_1^{1+\frac{1}{2\beta^*}},
	\]
	where \[\beta^*=\begin{cases}1+\frac{1}{\beta-1},\quad &{\rm for}\quad 3<\beta<\infty,\\
1,\quad &{\rm for}\quad \beta=\infty,	
\end{cases}\]
\[
p_1=\min\{p\in[m,0]:\Gamma(p)=\Gamma_{\max}\}
\]
and
\[
||\gamma||_\beta=\begin{cases}
	(\int_m^{0}|\gamma(r)|^{\beta}dr)^{\frac{1}{\beta}},\quad
	&{\rm for}\quad 3<\beta<\infty,\\
	\max\limits_{p\in{[m,0]}}|\gamma(p)|,\quad &{\rm for}\quad \beta=\infty,
\end{cases}
\]
then there exists $2\Gamma_{\max}+\frac{c_3}{2}>\xi>2\Gamma_{\max}$ such that $\mu(\xi)\leq-1$.
\end{lemma}}
{\begin{proof}
	Set
	\[
	E(k)=(2||\gamma||_\beta)^{\frac{3}{2}}\frac{k^2}{2k-1+\frac{3}{2\beta^*}}p_1^{\frac{3}{2\beta^*}-1}+(2||\gamma||_\beta)^{\frac{1}{2}}\frac{1}{2k+1+\frac{1}{2\beta^*}}p_1^{1+\frac{1}{2\beta^*}},\quad
	{\rm for}\quad k>\frac{1}{2}.\]
	Since \[\lim\limits_{k\to\frac{1}{2}}E(k)=\frac{\sqrt{ 2}}{3}\beta^*||\gamma||_{\beta}^{\frac{2}{3}}p_1^{\frac{3}{2\beta^*}-1}+\frac{2\sqrt2\beta^*}{1+4\beta^*}||\gamma||^{\frac{1}{2}}_\beta p_1^{1+\frac{1}{2\beta^*}},\] there exists  $k>\frac{1}{2}$ such that $\frac{(2\Gamma_{\max}-2\Gamma(p))^{\frac{3}{2}}c_1}{\Gamma_{\max}(Q-2c_2)}>E(k).$\\
	Let $n\geq2$ and \[
	\phi_n(p)=\begin{cases}
	0,\quad&{\rm for}\quad p_n\leq p\leq m,\\	
(p_n-p)^k,\quad&{\rm for}\quad 0\leq p\leq p_n.	
	\end{cases}
	\]
	where
	\[
	p_n=\frac{n-1}{n}p_1+\frac{m}{n}>0.
	\]
	Utilizing the H\"{o}lder's inequality, we have
	\[
a(p,2\Gamma_{\max})\leq\sqrt{2\Gamma_{\max}-2\Gamma(p)}\leq(2||\gamma||_\beta^{\frac{1}{2}})|p_1-p|^\frac{1}{2\beta^*}.
\]
Note that $\phi_n^2(0)=p,$ then we obtain
\[
\begin{split}
	&\int_0^1a^3(p,2\Gamma_{\max})(\phi_n'(p))^2dp+\int_0^1a(p,2\Gamma)\phi_n^2(p)dp\\
\leq&2\sqrt2||\gamma||_\beta^{\frac{3}{2}}k^2\{\int_0^{p_1}|p_1-p|^{\frac{3}{2\beta^*}}(p_n-p)^{2k-2}dp+\int_{p_1}^{p_n}|p_1-p|^\frac{3}{2\beta^*}(p_n-p)^{2k-2}dp\}\\
&+(2||\gamma||_\beta^\frac{1}{2})\{\int_{0}^{p_1}|p_1-p|^{\frac{1}{2\beta^*}}(p_n-p)^{2k}dp+\int_{p_1}^{p_n}|p_1-p|^{\frac{1}{2\beta^*}}(p_n-p)^{2k}dp\}\\
\leq&\phi_n^2(0)(A_n+B_n),
\end{split}
\]
where
\[
A_n=(2||\gamma_\beta||)^{\frac{3}{2}}\frac{k^2}{2k-1+\frac{3}{2\beta^*}}p_n^{\frac{3}{2\beta^*}-1}+(2||\gamma||_\beta)^\frac{1}{2}\frac{1}{2k+1+\frac{1}{2\beta^*}}p_n^{1+\frac{1}{2\beta^*}},
\]
and
\[
\begin{split}
B_n=&(\frac{1}{2k-1}-\frac{1}{2k-1+\frac{3}{2\beta^*}})(2||\gamma||)^\frac{3}{2}k^2\frac{(p_n-p_1)^{2k-1+\frac{3}{2\beta^*}}}{p_n^{2k}}\\
&+(\frac{1}{2k+1}-\frac{1}{2k+1+\frac{3}{2\beta^*}})(2||\gamma_\beta||)^\frac{1}{2}k^2\frac{(p_n-p_1)^{2k+1+\frac{3}{2\beta^*}}}{p_n^{2k}}.
\end{split}
\]
According to the definition of $\{p_n\}$, we have  $\frac{|p_n-p_1|}{p_n}\leq1$ and $\lim\limits_{n\to\infty} p_n=p_1.$
Hence for any given $\epsilon_0\in(0,\frac{(2\Gamma_{max}-2\Gamma(p))^\frac{3}{2}c_1}{\Gamma_{\max}(Q-2c_2)})$ there exists $N\in\mathbb N$ s.t. for all  $n>N$, following equations hold
\[A_n<\frac{(2\Gamma_{max}-2\Gamma(p))^\frac{3}{2}c_1}{\Gamma_{\max}(Q-2c_2)}-\epsilon,\quad B_n<\epsilon.\]
Hence, there exists $\phi_n\in \mathbb H^1(0,1)$ with $\phi_{n}(m)=0$, $\phi\not\equiv  0$ satisfying that \[
\int_0^1a^3(p,2\Gamma_{\max})(\phi_n'(p))^2dp+\int_0^1a(p,2\Gamma_{\max})\phi_n^2(p)dp<\frac{c_4^\frac{3}{2}c_1}{(2\Gamma_{\max}+\frac{c_3}{2})(Q-2c_2)}\phi_n^2(0)
.\]
Moreover \[
\mu(\xi)\leq\frac{-\frac{c_4^{\frac{3}{2}}c_1}{(2\Gamma_{\max}+\frac{c_3}{2})(Q-2c_2)}\zeta^2(0)+\int_0^1a^3\zeta_z^2dz}{\int_0^1a\zeta^2dz}\leq-1.
\]
 Note that the function $\xi\to\mu(\xi)$ is continuous, there exists $\xi>2\Gamma_{\max}$ such that $\mu(\xi)\leq-1.$
	\end{proof}}
\subsection{Local bifurcation}
In this section, we will apply Crandall-Rabinowitz bifurcation theorem (\cite{c-r}, pp325, Theorem 1.7) to study the local bifurcation of the non-trivial solutions of the blood flow model. Note that Crandall-Rabinowitz bifurcation theorem is a powerful tool to perform the bifurcation analysis (see e.g. \cite{Yi}).  In order to apply Crandall-Rabinowitz bifurcation theorem,
let $B_T=\{(w,z):w\in[-\pi,\pi],z=1\}\}$ be the top boundary and $B_B=\{(w,z):w\in[-\pi,\pi],z=0\}$ be the bottom boundary. And we define the Banach spaces
\[
\mathbb X=\{\chi\in C_{per}^{3,\alpha}(\bar{R}_z):\int_{-\pi}^{\pi}\chi(w,1)=d,\ \chi=0 \ {\rm on}\ B_B\},
\]
\[
\mathbb Y=C_{per}^{1,\alpha}(\bar R_z)\times C^{2,\alpha}_{per}(B_T).
\]
Moreover, system \eqref{h} can be rewritten as
\[
G(\xi,\chi)=0 \quad with\quad \chi\in \mathbb X,
\]
where $G:(\epsilon_0,1)\times \mathbb X\to \mathbb Y$ is given by{
\[
G_1(\xi,\chi)=(1+\chi_w^2)(H_{zz}+\chi_{zz})-2\chi_w\chi_{zw}(H_z+\chi_z)+(H_z+\chi_z)^2\chi_{ww}+\frac{\chi}{m}(H_z+\chi_z)^3
\]
and
\[
G_2(\xi,\chi)=1+\chi_w^2-\frac{1}{m^2}\left[Q-2F\right](H_z+\chi_z)^2\bigg|_{B_{B}}.
\]}
Combing  the equations above with \eqref{Q}, the linearized operator $G_{\chi}(\xi,0)=(G_{1\chi},G_{2\chi})\Bigg|_{\chi=0}$ is given by{
\begin{equation}\label{Ggamma}
	\begin{cases}
		G_{1\chi}(\xi,0)=\frac{\partial^2}{\partial z^2}+H^2_z\frac{\partial^2}{\partial w^2}+\frac{3\chi}{m}H^2_z\frac{\partial}{\partial z},\\
		G_{2\chi}(\xi,0)=[\frac{2f_2(q,0)}{\xi}-\frac{2}{m^2}(\xi m^2-2\int_{-d}^{0}f_2(q,r)dr)\cdot\frac{1}{\sqrt\xi}\cdot\frac{\partial}{\partial z}]\bigg|_{B_T}.
	\end{cases}
\end{equation}}
According to Crandall-Rabinowitz bifurcation theorem (\cite{c-r}, pp325, Theorem 1.7), in what follows, we divide the proof of the local bifurcation into two steps.
Firstly, we prove that the null space $ker\{G_{\chi}(\xi^\star,0)\}$ is one-dimensional. Secondly, we prove the transversality condition
$
[G_{\xi\chi}(\xi^*,0)](1,\zeta^*)\notin R(G_{\chi}(\xi^*,0))$ holds.
\begin{lemma}
	The null space $ker\{G_{\chi}(\xi^\star,0)\}$ is one-dimensional.
\end{lemma}
\begin{proof}
	Based on Lemma \ref{lemma1}, there exists a solution $\xi^{\star}$ of the equation $\mu(\xi)=-1$ with $\xi>2\chi_{\max}.$ Moreover, \eqref{Mk} has normalized ($M(1)=1$) eigenfunction  $M\in C^{3,\alpha}[0,1]$ such that $\mu(\xi)=-1$.  Then there exists at least one element $\chi(w,z)=M(z)\cos w$ in the space  $ker\{G_\chi(\xi^\star,0)\}$.\\
	By assuming that $m \in X$ is in the null space $ker{G_\chi(\xi^{\star}, 0)}$, we obtain that
	$m_k$ solve the equation \eqref{3.5}. Therefore, $m_1(z)$ is a constant multiple of $M(z)$. Moreover, if $m_k\ne0$ for some $k>2$,
	then
	{\[
	\frac{-\frac{a^3m_k^2(0)f_2(x,0)}{\xi(Q-2\int_{-d}^0f_2(x,r)dr)}+\int_0^1a^3(\frac{\partial m_k}{\partial z})^2dz}{\int_0^1am_k^2dz}<-1.
	\]}
	This is a contradiction, because a minimum is attained corresponding to $\mu(\xi^\star) = -1$. And
	from the differential equation and from the boundary condition at $z = 0$ in \eqref{M0} with $k = 0$ we
	get that
	\[
	m_0(z)=A_0\int_0^za^{-3}(r,\xi^{\star})dr,\quad z\in[0,1]
	\]
	holds for some $A_0\in \mathbb R$. Note that  $\int_\pi^\pi m_0(1)dw=2\pi m_0(1)=0$ and $A_0\int_0^{1}a^{-3}(z,\xi^\star) dz=0.$ And it is not difficult to show that $\int_0^{1}a^{-3}(z,\xi^{\star})dz\ne 0$ by direct computation. Hence $A_0=0$ and $m_0=0$.
\end{proof}

\begin{lemma}
	The following transversality condition holds.
	\[
	[G_{\xi\chi}(\xi^*,0)](1,\zeta^*)\notin R(G_{\chi}(\xi^*,0)).
	\]
\end{lemma}{
\begin{proof}
	Keep in mind that $a_z=-\chi(p)a^{-1}$,
	where $a=a(\cdot,\xi^*). $\\ It follows from \eqref{Ggamma} that
	\[
	G_{\xi \chi}(\xi^*,0)=-\left( a^{-4}\frac{\partial^2}{\partial w^2}+3a_za^{-3}\frac{\partial}{\partial z},\left(\frac{2f_2(q,0)}{(\xi^*)^2}+(\xi^*)^{-\frac{1}{2}}\frac{\partial}{\partial z}+\frac{2\int_{-d}^{0}f_2(q,r)dr}{m^2}\cdot(\xi^*)^{-\frac{3}{2}}\frac{\partial}{\partial z}\right)\bigg|_{B_T}\right).
	\]
	Lemma is equivalent to proof
	\begin{equation}\label{ne}
		\begin{split}
		&\iint_{R_z}a^3\zeta^*2\xi^*(a^{-4}\frac{\partial^2\zeta^*}{\partial w^2}+3a^{-3}a_z\frac{\partial\zeta^*}{\partial z})dwdz
		\\&+\int_{B_T}[2(\zeta^*)^2f_2(q,0)+(\xi^*)^{\frac{3}{2}}\zeta^{*}\zeta_z^{*}+\frac{2\int_{-d}^0f_2(q,r)dr}{m^2}(\xi^*)^{\frac{1}{2}}\zeta^*\zeta^*_z]dw\ne0.
		\end{split}
	\end{equation}
	Note that $\zeta^*(w,z)=M(z)\cos(w)$. Since $M\ne0$ solves \eqref{Mk} with $\mu=-1$, and $a(1) = \xi^*$, then we obtain
	\begin{equation}
		\begin{cases}
			a^3\zeta_{zz}^*+3a^2a_{z}\zeta_z^*-a\zeta^*=0,\quad&{\rm for}\quad 0<z<1,\\
			\xi^*(Q-2c_3)\zeta^*_z=c_2\zeta^*,\quad&{\rm on }\quad z=1,\\
			\zeta^*=0,\quad &{\rm on}\quad z=0.
		\end{cases}
	\end{equation}
	Hence, we obtain
	\[
	\begin{split}
		\iint_{R_z}a_z\zeta^*\zeta_z^*dzdw&=\int_{B_T}a\zeta^*\zeta_z^*dw-\iint_{R_z}a\zeta^*\zeta_{zz}^*dwdz-\iint_{R_z}a(\zeta^*_z)^2dwdz\\
		&=\int_{B_T}a\zeta^*\zeta^*_zdw+\iint_{R_z}\zeta^*(3a_z\zeta^*_z-a^{-1}\zeta^*)dwdz-\iint_{R_z}a(\zeta^*_z)^2dwdz.
	\end{split}
	\]
	Furthermore,
	\[
	\iint_{R_z} a_z\zeta^*\zeta^*_zdwdz=-\frac{1}{2}\int_{B_T}a\zeta^*\zeta^*_zdw+\frac{1}{2}\iint_{R_z}\left({a^{-1}(\zeta^*)^2+a(\zeta_z^*)^2}\right)dwdz.
	\]
	Keep in mind $\zeta_{ww}=-\zeta$ and $a(1)=\xi$. And form \eqref{ne}, we need to show that
	\[
	-2\xi^*\iint_{R_z}a^{-1}(\zeta^*)^2 dwdz+3\xi^*\iint_{R_z}\left({a^{-1}(\zeta^*)^2+a(\zeta_z^*)^2}\right)dwdz+\int_{B_T}2c_2(\zeta^*)^2dw-2\xi^*\int_{B_T}a\zeta^*\zeta^*_zdw\ne 0,
	\]
	which is equal to prove
	\[
	\begin{split}
	\iint_{R_z}\xi^* &a^{-1}{\zeta^*}^2dwdz+3\xi^*\iint_{R_z}a(\zeta_z^2)dwdz-3\int_{B_T}\frac{ac_2(\zeta^*)^2}{\frac{Q}{2}-c_3}dw
	\\+&\int_{B_T}2(\zeta^*)^2c_1dw+\int_{B_T}\frac{(\zeta^*)^{\frac{1}{2}}(\zeta^*)^2c_2}{\frac{Q}{2}-c_3}dw+\int_{B_T}\frac{2c_2^2\xi^{-\frac{1}{2}}(\zeta^*)^2}{m^2(\frac{Q}{2}-c_3)}dw\neq0.
	\end{split}
	\]
	With the  average inequality and Cauchy-Schwartz inequality, we obtain
	\[
	\begin{split}
		&\iint_{R_z}\xi^* a^{-1}{\zeta^*}^2dwdz+3\xi^*\iint_{R_z}a(\zeta^*_z)^2dwdz-3\int_{B_T}\frac{ac_2(\zeta^*)^2}{\frac{Q}{2}-c_3}dw
		\\&+\int_{B_T}2(\zeta^*)^2c_1dw+\int_{B_T}\frac{(\zeta^*)^{\frac{1}{2}}(\zeta^*)^2c_2}{\frac{Q}{2}-c_3}dw+\int_{B_T}\frac{2c_2^2\xi^{-\frac{1}{2}}(\zeta^*)^2}{m^2(\frac{Q}{2}-c_3)}dw\\
		\geq& 2\sqrt3\xi^*\left\{\iint_{R_z} a^{-1}{\zeta^*}^2dwdz\right\}^{\frac{1}{2}}\left\{\iint_{R_z}a(\zeta^*_z)^2dwdz\right\}^{\frac{1}{2}}+3\int_{B_T}\frac{ac_2(\zeta^*)^2}{\frac{Q}{2}-c_3}dw
		\\&+\int_{B_T}2(\zeta^*)^2c_1dw+\int_{B_T}\frac{(\zeta^*)^{\frac{1}{2}}(\zeta^*)^2c_2}{\frac{Q}{2}-c_3}dw+\int_{B_T}\frac{2c_2^2(\xi^*)^{-\frac{1}{2}}(\zeta^*)^2}{m^2(\frac{Q}{2}-c_3)}dw\\
		\geq& \sqrt3\xi^*\int_{B_T}({\zeta^*})^2dw-3\int_{B_T}\frac{ac_2(\zeta^*)^2}{\frac{Q}{2}-c_3}dw
+\int_{B_T}2(\zeta^*)^2c_1dw
\\&+\int_{B_T}\frac{(\zeta^*)^{\frac{1}{2}}(\zeta^*)^2c_2}{\frac{Q}{2}-c_3}dw+\int_{B_T}\frac{2c_2^2(\xi^*)^{-\frac{1}{2}}(\zeta^*)^2}{m^2(\frac{Q}{2}-c_3)}dw\\
		=&\left\{(\sqrt3-\frac{3c_2}{\frac{Q}{2}-c_3})\xi^*+\frac{c_2}{\frac{Q}{2}-c_3}(\xi^*)^{\frac{1}{2}}+\frac{2c_2^2}{m^2(\frac{Q}{2}-c_3)}\cdot\frac{1}{(\xi^*)^{\frac{1}{2}}}+2c_1\right\}\int_{B_T}({\zeta^*})^2dw>0.
	\end{split}
	\]
	The last item is larger than 0, guaranteed by the condition $Q>2c_3+2\sqrt3c_2$.
\end{proof}}

Through the above analysis, the main result of this paper is proven. We conclude it as follows.
\begin{theorem}\label{main result}
	There exists a $C^1$-curve $C_{loc}$ of
	small-amplitude solution $h \in C^{3,\alpha}_
	{per}(\bar R_z)$ and $(u,v,\Omega) \in C^{2,\alpha}_
	{per}(\bar{R}_z)\times C^{2,\alpha}_
	{per}( \bar{R}_z) \times C^{3,\alpha}_
	{per}(R)$. Moreover, the solution curve $C_{loc}$ contains precisely one function that is independent of $w$.
\end{theorem}

\section*{Conflict of Interest}
\hskip\parindent
The authors declare that they have no conflict of interest.

\section*{Data Availability Statement}
No data was used for the research in this article.

\section*{Contributions}
We declare that the authors are ranked in alphabetic order of their names and all of them have the same contributions to this paper.

\end{document}